\documentclass[10pt]{amsart}
\usepackage[latin1]{inputenc}
\usepackage{color}
\usepackage{enumerate}
\usepackage{amssymb}
\usepackage[all]{xy}
\usepackage{hyperref}
\usepackage{verbatim} 
\usepackage{mathrsfs}

\newtheorem{thm}{Theorem}[section]
\newtheorem{cor}[thm]{Corollary}
\newtheorem{lem}[thm]{Lemma}
\newtheorem{prop}[thm]{Proposition}

\theoremstyle{definition}
\newtheorem{defn}[thm]{Definition}
\newtheorem{rem}[thm]{Remark}
\newtheorem{fact}[thm]{Fact}
\newtheorem{prob}[thm]{Problem}
\newtheorem{example}[thm]{Example}
\numberwithin{equation}{section}
\newcommand{\norm}[1]{\left\Vert#1\right\Vert}
\newcommand{\abs}[1]{\left\vert#1\right\vert}
\newcommand{\set}[1]{\left\{#1\right\}}

\newcommand{\To}{\longrightarrow}

\newcommand{\nat}{\mathbb{N}}

\usepackage{color}
\usepackage{enumerate}
\newcommand{\sub}{\subseteq}

\begin{document}

\setcounter{tocdepth}{1}


\title{A Godefroy-Kalton principle for free Banach lattices}

\author[A.\ Avil\'es]{Antonio Avil\'es}
\address{Universidad de Murcia, Departamento de Matem\'{a}ticas, Facultad de Matem\'{a}ticas, 30100 Espinardo (Murcia), Spain.}
\email{avileslo@um.es}

\author[G. Mart\'inez-Cervantes]{Gonzalo Mart\'inez-Cervantes}
\address{Universidad de Murcia, Departamento de Matem\'{a}ticas, Facultad de Matem\'{a}ticas, 30100 Espinardo (Murcia), Spain.}
\email{gonzalo.martinez2@um.es}

\author[J. Rodr\'{i}guez]{Jos\'e Rodr\'{i}guez}
\address{Departamento de Ingenier\'{i}a y Tecnolog\'{i}a de Computadores,
Facultad de Inform\'{a}tica, Universidad de Murcia, 30100 Espinardo (Murcia), Spain.}  
\email{joserr@um.es}

\author[P. Tradacete]{Pedro Tradacete}
\address{Instituto de Ciencias Matem\'aticas (CSIC-UAM-UC3M-UCM)\\
Consejo Superior de Investigaciones Cient\'ificas\\
C/ Nicol\'as Cabrera, 13--15, Campus de Cantoblanco UAM\\
28049 Madrid, Spain.}
\email{pedro.tradacete@icmat.es}


\begin{abstract}
Motivated by the Lipschitz-lifting property of Banach spaces introduced by Godefroy and Kalton, we consider the lattice-lifting property, which is an analogous notion within the category of Banach lattices and lattice homomorphisms. Namely, a Banach lattice $X$ satisfies the lattice-lifting property if every lattice homomorphism to~$X$ having a bounded linear right-inverse must have a lattice homomorphism right-inverse. In terms of free Banach lattices, this can be rephrased into the following question: 
which Banach lattices embed into the free Banach lattice which they generate as a lattice-complemented sublattice? We will provide necessary conditions for a Banach lattice to have the lattice-lifting property, and show that this property is shared by Banach spaces with a $1$-unconditional basis as well as free Banach lattices. The case of $C(K)$ spaces will also be analyzed.
\end{abstract}

\maketitle

\section{Introduction}

In a fundamental paper concerning the Lipschitz structure of Banach spaces, G.~Godefroy and N.J.~Kalton introduced the Lipschitz-lifting property of a Banach space. 
In order to properly introduce this notion, and as a motivation for our work, let us recall the basic ingredients for this construction (see \cite{GK} for details). 
Given a Banach space~$E$, let $Lip_0(E)$ denote the Banach space of all real-valued Lipschitz functions on~$E$ which vanish at $0$, equipped with the norm 
$$
	\|f\|_{Lip_0(E)}=\sup\left\{\frac{|f(x)-f(y)|}{\|x-y\|_E}: \, x,y\in E, \, x\neq y\right\}.
$$ 
The Lipschitz-free space over~$E$, denoted by~$\mathcal F(E)$, is the canonical predual of $Lip_0(E)$, that is the closed linear span of the evaluation 
functionals $\delta(x)\in Lip_0(E)^*$ given by $\langle\delta(x),f\rangle=f(x)$ for all $x\in E$ and all $f\in Lip_0(E)$. In this setting, the map 
$$
\delta:E\rightarrow \mathcal F(E)
$$ 
is easily seen to be a non-linear isometry and there is a bounded linear operator, called the barycenter map, 
$$
	\beta:\mathcal F(E)\rightarrow E
$$ 
satisfying $\beta \circ \delta=id_E$. A Banach space $E$ is said to have the Lipschitz-lifting property if there is a bounded \emph{linear} operator $T:E\rightarrow \mathcal F(E)$ 
such that $\beta \circ T=id_E$. In particular, Banach spaces with the Lipschitz-lifting property embed linearly into their corresponding Lipschitz-free spaces. 
It was proved in \cite[Theorem~3.1]{GK} that every separable Banach space has the Lipschitz-lifting property, and as a consequence if a separable Banach space 
$E$ embeds isometrically into a Banach space~$F$ (via a not necessarily linear map), then $F$ contains a linear subspace which is linearly isometric to~$E$ (see \cite[Corollary~3.3]{GK}). This is a relevant step in the non-linear theory of Banach spaces and in what is called the Ribe program, whose aim is to get a better understanding of the relation between the metric and linear structure of Banach spaces and their subsets (see \cite{BL,K,Naor12} for further background and surveys of recent 
developments).

Our aim here is to study the analogous situation when considering the category of Banach lattices and lattice homomorphisms. 
Let us recall that free and projective Banach lattices were introduced by B.~de~Pagter and A.W.~Wickstead in~\cite{dPW15}. 
The free Banach lattice generated by a set $A$, denoted by $FBL(A)$, was 
further extended in~\cite{ART18} to define the concept of free Banach lattice generated by a Banach space~$E$, denoted by~$FBL[E]$. 
This is a Banach lattice together with a linear isometry $\delta_E: E\to FBL[E]$ such that for every Banach lattice $X$ and every bounded linear operator 
$T : E \to X$ there is a unique lattice homomorphism $\hat T : FBL[E] \to X$ such that $\hat T \circ \delta_E = T$ and moreover $\|\hat T\| = \norm{T}$. 
Existence and uniqueness up to Banach lattice isometries of free Banach lattices were proved in~\cite{ART18} (also in~\cite{dPW15} for the case of~$FBL(A)$), 
and an explicit description was provided in \cite[Theorem 2.5]{ART18}, as follows.
Let $E$ be a Banach space. Given a positively homogeneous function $f:E^\ast \to \mathbb{R}$ (i.e., $f(\lambda x ^*)=\lambda f(x^*)$ for every $\lambda >0$ and 
every $x^* \in E^*$), we consider 
$$ 
	\norm{f}_{FBL[E]} := \sup \set{\sum_{i = 1}^n \abs{f(x_{i}^{*})} : n \in \mathbb{N}, \, 
	x_1^{*}, \ldots, x_n^{*} \in E^{*},\ \sup_{x \in B_E} \sum_{i=1}^n \abs{x_i^{*}(x)} \leq 1 } \in [0,\infty].
$$
We denote by $H_0[E]$ the linear space of all positively homogeneous functions $f:E^*\to \mathbb R$ such that $\|f\|_{FBL[E]}<\infty$,
which becomes a Banach lattice when equipped with the norm $\|\cdot\|_{FBL[E]}$ and the pointwise order.
Then the Banach lattice $FBL[E]$ is the sublattice of $H_0[E]$ generated by the set $\{\delta_E(x):x\in E\}$,
where $\delta_E(x)$ denotes the evaluation map defined by
$$
	\delta_E(x): E^*\to \mathbb R,
	\quad \delta_E(x)(x^\ast):= x^\ast(x)
	\text{ for all }x^*\in E^*.
$$ 
These evaluations form the natural copy of~$E$ inside $FBL[E]$, via the linear isometry $\delta_E: E\to FBL[E]$.

It is worth noticing that in the case of the underlying Banach space $E$ being a Banach lattice, its lattice structure seems to be ``forgotten'' when it embeds into $FBL[E]$. However, some traces of this structure can still be recovered in some situations. To clarify this, let $X$ be a Banach lattice, and consider the identity map $id_X:X\rightarrow X$ that can 
be ``extended'' to a lattice homomorphism from $FBL[X]$ to~$X$ which, in analogy with the Lipschitz-free space situation, 
will be denoted by $\beta_X$. That is, $\beta_X:FBL[X]\to X$ is the unique lattice homomorphism such that $\beta_X \circ \delta_X=id_X$, and we have $\|\beta_X \|=1$.
Our purpose in this paper is to explore under which conditions there exists a lattice homomorphism $T: X\to FBL[X]$ such that $\beta_X \circ T=id_X$. Note that,
if this is the case, then $X$ is lattice isomorphic (via~$T$) to a sublattice of $FBL[X]$ which is lattice-complemented (via~$\beta_X$).

\begin{defn}
We say that a Banach lattice $X$ has the {\em lattice-lifting} property if 
there exists a lattice homomorphism $T: X\to FBL[X]$ such that $\beta_X \circ T=id_X$. If in addition $T$ can be chosen with $\|T\|=1$, then
we say that $X$ has the {\em isometric lattice-lifting} property.
\end{defn}

The paper is organized as follows. In Section \ref{s:llp}, we present some basic properties and examples of Banach lattices with the lattice-lifting property:
these include projective Banach lattices and the space $FBL[E]$ for any Banach space~$E$. 
As a consequence of the latter, a Banach lattice $X$ has the lattice-lifting property if and only if it is lattice isomorphic to a lattice-complemented 
sublattice of~$FBL[X]$ (Corollary~\ref{cor:char}).
In the first part of Section~\ref{section:CK}
we discuss the lattice-lifting property for the Banach lattice $C(K)$, where $K$ is a compact Hausdorff topological space. 
It turns out that the lattice-lifting property of $C(K)$ implies that $K$ is a neighborhood retract of the closed dual unit ball $B_{C(K)^*}$
(equipped with the $w^*$-topology), see Theorem~\ref{thm:CK}. This, combined with a result of~\cite{AMR20_jmaa}, ensures
that, for metrizable~$K$, the lattice-lifting property of~$C(K)$ is equivalent to being a projective Banach lattice (Corollary~\ref{cor:CK}).
In the second part of Section~\ref{section:CK} we analyze some topological properties of the set 
of all real-valued lattice homomorphisms on a Banach lattice having the isometric lattice-lifting property. 
Finally, in Section \ref{section:unconditional} we show that every Banach space with a $1$-unconditional basis (as a Banach lattice with 
the coordinatewise order) satisfies the isometric lattice-lifting property (Theorem~\ref{t:unconditional}). This is the most technical result
of the paper and its proof is based on some ideas of~\cite{AMR20}, where the particular case of~$c_0$ was addressed.

\subsection*{Terminology}

All our Banach spaces are real. Given a Banach space~$E$, its norm is denoted by $\|\cdot\|_E$ or simply~$\|\cdot\|$. We write  
$B_E:=\{x\in E:\|x\|\leq 1\}$ for the closed unit ball of~$E$ and the symbol~$E^*$ stands for the topological dual of~$E$. The weak$^*$-topology
on~$E$ is denoted by~$w^*$. 

By a {\em lattice homomorphism} between Banach lattices we mean a bounded linear operator preserving lattice operations.
Given a Banach lattice $X$, we write $X_+:=\{x\in X: x\geq0\}$. By a ``sublattice'' of~$X$ we mean a closed vector sublattice. 
We say that a sublattice $Y \sub X$ is {\em lattice-complemented} if there is a lattice homomorphism $P:X\to Y$ such that $P|_Y=id_Y$. 
We denote by $Hom(X,\mathbb R) \sub X^*$ the set of all lattice homomorphisms from~$X$ to~$\mathbb R$. In the case
of free Banach lattices over Banach spaces this set admits a concrete description. Namely, given a Banach space~$E$, the set $Hom(FBL[E],\mathbb R)$ 
is precisely $\{\varphi_{x^*}:x^*\in E^*\}$, where each functional $\varphi_{x^*}:FBL[E]\to \mathbb R$ is defined by
$$
	\varphi_{x^*}(f):=f(x^*)
	\quad\text{for every $f\in FBL[E]$,} 
$$
see \cite[Corollary~2.7]{ART18}.

By a ``compact space'' we mean a compact Hausdorff topological space. 
Given a compact space~$K$, we denote by $C(K)$ 
the Banach lattice of all real-valued continuous functions on~$K$, equipped with the supremum norm and the pointwise order.

Recall that a Schauder basis $(e_n)_{n\in \nat}$ of a Banach space~$E$ is called a {\em $1$-unconditional basis} if
for every $n\in \nat$, and $a_1,\dots,a_n,b_1,\dots,b_n\in \mathbb R$ satisfying $|a_i|\leq |b_i|$ for $i=1,\dots,n$, then
$$
	\left\|\sum_{i=1}^n a_i e_i\right\|\leq \left\|\sum_{i=1}^n b_i e_i\right\|.
$$
In this case, $E$ becomes a Banach lattice when equipped with the coordinatewise order given by
$$
	x\leq y \quad \Longleftrightarrow \quad e_n^*(x) \leq e_n^*(y) \text{ for all $n\in \nat$},
$$
where $(e_n^*)_{n\in \mathbb N}$ is the sequence in~$E^*$ of biorthogonal functionals associated with~$(e_n)_{n\in \mathbb N}$.

\section{Basics on the lattice-lifting property}\label{s:llp}

The following proposition shows some statements which are easily seen to be equivalent to the lattice-lifting property:

\begin{prop}\label{prop:EquivalencesLLP}
Let $X$ be a Banach lattice and $\lambda\geq 1$. The following statements are equivalent:
\begin{enumerate}
\item[(i)] There is a lattice homomorphism $\alpha: X \to FBL[X]$ with $\|\alpha\|\leq \lambda$ such that $\beta_X \circ \alpha=id_X$.
\item[(ii)] For every commutative diagram
$$
 			\xymatrix@R=3pc@C=3pc{Z_1 \ar[r]^{\pi} & Z_2\\
			 & X \ar@{->}[ul]_{T} \ar[u]_{\phi} \\ }
$$ 
where $Z_1$ and $Z_2$ are Banach lattices, $\pi$ and $\phi$ are lattice homomorphisms and $T$ is a bounded linear operator, 
there is also a lattice homomorphism $T':X\to Z_1$ with $\pi\circ T' = \phi$ and $\|T'\|\leq \lambda \|T\|$.
\item[(iii)] If $Y$ is a Banach lattice, $\pi:Y\to X$ is a lattice homomorphism and there is a bounded linear operator $T: X\to Y$ such that $\pi\circ T=id_X$, 
then there is also a lattice homomorphism $T':X\to Y$ such that $\pi\circ T'=id_X$ and $\|T'\|\ \leq \lambda \|T\|$.
\end{enumerate}
\end{prop}
\begin{proof}
(i)$\Rightarrow$(ii) Let $\hat{T}:FBL[X]\to Z_1$ be the unique lattice homomorphism satisfying $\hat{T}\circ \delta_X=T$. 
Since $\pi \circ \hat{T}$ and $\phi \circ \beta_X$ are lattice homomorphisms and $\pi \circ \hat{T}\circ\delta_X=\phi=\phi \circ \beta_X \circ \delta_X$,
we have $\pi \circ \hat{T}=\phi \circ \beta_X$ and so $T':=\hat{T}\circ \alpha$ satisfies the required properties.

(ii)$\Rightarrow$(iii) Observe that (iii) is the particular case of~(ii) when $Z_1=Y$, $Z_2=X$ and $\phi=id_X$. 
	
(iii)$\Rightarrow$(i) Observe that (i) is the particular case of~(iii) when $Y=FBL[X]$, $\pi=\beta_X$ and $T=\delta_X$.	
\end{proof}

We say that a Banach lattice $X$ has the lattice-lifting property \textit{with constant $\lambda$} if $X$ and $\lambda$ satisfy any of the equivalent 
conditions of Proposition~\ref{prop:EquivalencesLLP}.

Projective Banach lattices provide a source of examples satisfying the lattice-lifting property, see Proposition~\ref{prop:projective} below. 
Given $\lambda\geq 1$, a Banach lattice $X$ is called {\em $\lambda$-projective} if whenever $Y$ is a Banach lattice, $J$ is a closed ideal of~$Y$ 
and $Q: Y \rightarrow Y/J$ is the quotient map, then for every lattice homomorphism $T : X\rightarrow Y/J$ there is a lattice homomorphism $\hat T : X \rightarrow Y$ such that $T = Q \circ\hat T$ and $\|\hat T\|\leq \lambda \|T\|$.
A Banach lattice is called {\em projective} if it is $\lambda$-projective for some $\lambda\geq 1$. 
This class of Banach lattices has been studied in \cite{AMR20_jmaa,AMR20,dPW15}. 
Examples of projective Banach lattices include all finite dimensional Banach lattices, 
all free Banach lattices of the form $FBL(A)$ and the spaces $\ell_1$ and $C[0,1]$, see \cite[\S10 and \S11]{dPW15}.

\begin{prop}\label{prop:projective}
Every $\lambda$-projective Banach lattice has the lattice-lifting property with constant $\lambda$.
\end{prop}
\begin{proof}
Suppose $X$ is a $\lambda$-projective Banach lattice for some constant~$\lambda \geq 1$. 
Since $\beta_X:FBL[X]\rightarrow X$ is a surjective lattice homomorphism, it factors as $\beta_X=S\circ Q$, 
where 
$$
	Q: FBL[X]\rightarrow FBL[X]/\ker(\beta_X)
$$ 
is the canonical quotient map and
$$
	S: FBL[X]/\ker(\beta_X) \rightarrow X
$$ 
is a lattice isometry. Now, we consider the lattice isometry $T:=S^{-1}: X\rightarrow FBL[X]/\ker(\beta_X)$. By hypothesis, there exists a lattice homomorphism 
$\alpha:X\rightarrow FBL[X]$ such that $T=Q \circ \alpha$ and $\|\alpha\|\leq \lambda$. 
Therefore, we have $\beta_X\circ \alpha=S\circ Q \circ \alpha=S \circ T = id_X$, as required. 
\end{proof}

Projectivity for free Banach lattices is a quite restrictive property, namely: if $E$ is a Banach space without the Schur property, then $FBL[E]$ is not projective, 
see~\cite[Theorem 1.3]{AMR20_jmaa}. However, all free Banach lattices satisfy the isometric lattice-lifting property (this is entirely analogous to \cite[Lemma 2.10]{GK}):

\begin{prop}\label{prop:FBL}
$FBL[E]$ has the isometric lattice-lifting property for every Banach space $E$.
\end{prop}
\begin{proof}
Consider the isometric embedding 
$$
	T:=\delta_{FBL[E]}\circ \delta_E: E\rightarrow FBL[FBL[E]]
$$
and the lattice homomorphism $\hat T:FBL[E]\rightarrow FBL[FBL[E]]$ satisfying $\hat T \circ \delta_E=T$ and $\|\hat T\|=\|T\|=1$. 
We claim that $\beta_{FBL[E]} \circ \hat T=id_{FBL[E]}$. Indeed, given any $x\in E$ we have
$$
	(\beta_{FBL[E]} \circ \hat T) (\delta_E(x))=\beta_{FBL[E]}(T(x))=\beta_{FBL[E]}\big(\delta_{FBL[E]}(\delta_E(x))\big)= \delta_E(x).
$$
Since $\beta_{FBL[E]}\circ \hat T$ and $id_{FBL[E]}$ are lattice homomorphisms and the sublattice 
generated by $\{\delta_E(x):x\in E\}$ is~$FBL[E]$, it follows that $\beta_{FBL[E]} \circ \hat T= id_{FBL[E]}$.
\end{proof}

Any Banach lattice $X$ having the lattice-lifting property is lattice isomorphic to a sublattice of~$FBL[X]$.
Therefore, any property of free Banach lattices which is inherited by sublattices holds for Banach lattices having the lattice-lifting property. Corollary~\ref{cor:necessary} below
is obtained via this argument when applied to a couple of Banach lattice properties.

A Banach lattice $X$ is said to satisfy the {\em $\sigma$-bounded chain condition} if $X_+\setminus\{0\}$ can be written as a countable union 
$X_+\setminus \{0\}=\bigcup_{n\geq 2} \mathcal F_n$ such that, for each $n\geq 2$ and each $\mathcal G\sub \mathcal F_n$ with cardinality~$n$, there exist distinct
$x,y\in \mathcal G$ such that $x\wedge y\neq 0$. The $\sigma$-bounded chain condition
is stronger than the {\em countable chain condition} (every uncountable subset of~$X_+$ contains two distinct elements which are not disjoint).

\begin{rem}\label{rem:sigma-bcc}
Let $X$ be a Banach lattice for which there is a sequence $(\phi_n)$ in $Hom(X,\mathbb R)$ which separates 
the points of~$X$.
Then $X$ has the $\sigma$-bounded chain condition. Indeed, it suffices to consider 
$$
	\mathcal F_n:=\{x\in X_+: \ \phi_m(x)\neq 0 \text{ for some }m<n\}
$$
for all $n\geq 2$. In particular, any sublattice of~$\ell_\infty$ has the $\sigma$-bounded chain condition.
\end{rem}

\begin{cor}\label{cor:necessary}
If a Banach lattice $X$ has the lattice-lifting property, then:
\begin{enumerate}
\item[(i)] $Hom(X,\mathbb R)$ separates the points of~$X$.
\item[(ii)] $X$ satisfies the $\sigma$-bounded chain condition.
\end{enumerate}
\end{cor}
\begin{proof}
For any Banach space $E$, the set $Hom(FBL[E],\mathbb R)$ separates the points of~$FBL[E]$ 
(bear in mind the description of $Hom(FBL[E],\mathbb R)$ given in the introduction)
and $FBL[E]$ satisfies the $\sigma$-bounded chain condition, see \cite[Theorem 1.2]{APR}.
Both properties are clearly inherited by sublattices.
\end{proof}

\begin{example}\label{example:Lp}
\begin{enumerate}
\item[(i)] $L_p(\mu)$ fails the lattice-lifting property for any $1\leq p\leq \infty$ whenever the measure $\mu$ is not purely atomic,
because in this case $Hom(L_p(\mu),\mathbb R)$ does not separate those functions supported on the non-atomic part of the measure space.
\item[(ii)] $c_0(\Gamma)$ and $\ell_p(\Gamma)$ fail the lattice-lifting property for any $1\leq p\leq \infty$ whenever $\Gamma$ is an uncountable set,
because such Banach lattices fail the countable chain condition.
\end{enumerate}
\end{example}

In Example~\ref{example:torus} and Corollary~\ref{cor:sublatticefails} below we will show that, in general, the lattice-lifting property is not inherited
by sublattices. However, for lattice-complemented sublattices the situation is different:

\begin{prop}\label{prop:sublatticecomplemented}
Let $X$ be a Banach lattice having the lattice-lifting property. If $Y\sub X$ is a lattice-complemented sublattice, then $Y$ 
has the lattice-lifting property.
\end{prop}
\begin{proof} Let $j:Y\rightarrow X$ denote the identity embedding and let $P:X\rightarrow Y$ be a lattice homomorphism such that $P|_Y=id_Y$. 
Let $\alpha_X:X\rightarrow FBL[X]$ be a lattice homomorphism such that $\beta_X\circ \alpha_X=id_X$. 
Let $S:FBL[X]\rightarrow FBL[Y]$ be the unique lattice homomorphism satisfying $S\circ \delta_X= \delta_Y \circ P$. 
We claim that the lattice homomorphism $\alpha_Y:=S \circ \alpha_X \circ j:Y\rightarrow FBL[Y]$ satisfies $\beta_Y\circ \alpha_Y=id_Y$. 
Indeed, since both $\beta_Y \circ S$ and $P \circ \beta_X$ are lattice homomorphisms and 
$$
	\beta_Y \circ S \circ \delta_X=\beta_Y \circ \delta_Y \circ P = P = P \circ \beta_X \circ \delta_X,
$$
we have $\beta_Y \circ S=P \circ \beta_X$. Hence
$$
	\beta_Y \circ \alpha_Y=\beta_Y \circ S \circ \alpha_X \circ j=P\circ \beta_X \circ \alpha_X \circ j=id_Y,
$$
as claimed.
\end{proof}

As a consequence of Propositions~\ref{prop:FBL} and~\ref{prop:sublatticecomplemented} we have:

\begin{cor}\label{cor:char}
A Banach lattice $X$ has the lattice-lifting property if and only if it is lattice isomorphic to a lattice-complemented 
sublattice of~$FBL[X]$.
\end{cor}

\begin{rem}\label{rem:quotients}
In general, the lattice-lifting property is not inherited by lattice quotients. For instance, the space $L_1[0,1]$ 
fails the lattice-lifting property (see Example~\ref{example:Lp})
and is a lattice quotient of~$FBL[\ell_1]$ (which has the lattice-lifting property, by Proposition~\ref{prop:FBL}). 
Indeed, since $L_1[0,1]$ is a separable Banach space, there is a surjective bounded linear operator 
$T:\ell_1\to L_1[0,1]$, hence the unique lattice homomorphism $\hat T: FBL[\ell_1] \to L_1[0,1]$ satisfying $\hat T \circ \delta_{\ell_1}=T$ is
surjective as well.
\end{rem}

\section{Retracts and the lattice-lifting property}\label{section:CK}

Given a topological space~$T$, a subspace $A \sub T$ said to be a {\em retract} (resp., {\em neighborhood retract}) 
of~$T$ if there is a continuous map from $T$ to~$A$ which is the identity on~$A$ (resp., there is an open set $U \sub T$ containing~$A$ such that
$A$ is a retract of~$U$). Such a map is called a {\em retraction}. 
Absolute neighborhood retracts play an important role in the study of projective Banach lattices (see \cite{dPW15} and~\cite{AMR20_jmaa}). They
also appear in a natural way when dealing with the lattice-lifting property. 
Recall that a compact space is said to be an {\em absolute retract} or {\em AR} (resp., {\em absolute neighborhood retract} or {\em ANR})
if it is a retract (resp., neighborhood retract) of any compact space containing it.
The following elementary characterization will be useful:

\begin{fact}\label{fact:ARandANR}
A compact space $K$ is an AR (resp. an ANR) if and only if it is homeomorphic to a retract (resp., neighborhood retract) 
of~$[0,1]^\Gamma$ for some non-empty set~$\Gamma$.
\end{fact}

The main result of this section is:

\begin{thm}\label{thm:CK}
Let $K$ be a compact space such that $C(K)$ has the lattice-lifting property. Then $K$ is a neighborhood retract of~$(B_{C(K)^\ast},w^*)$.
If in addition $K$ is metrizable, then $K$ is an ANR.
\end{thm}

As usual, we identify $K$ with its canonical homeomorphic copy inside $B_{C(K)^*}$ (with the $w^*$-topology), which consists
of all evaluation functionals of the form $ h \mapsto h(t)$ for every $h\in C(K)$, where $t\in K$.

We stress that, for an arbitrary compact space~$K$, the space $C(K)$ is projective whenever $K$ is an
ANR, see~\cite[Theorem 1.4]{AMR20_jmaa}. So, bearing in mind Proposition~\ref{prop:projective} and Theorem~\ref{thm:CK}, we can summarize the
case of metrizable compacta as follows.

\begin{cor}\label{cor:CK}
Let $K$ be a compact metrizable space. The following statements are equivalent:
\begin{enumerate}
\item[(i)] $K$ is an ANR.
\item[(ii)] $C(K)$ is projective.
\item[(iii)] $C(K)$ has the lattice-lifting property. 
\end{enumerate}
\end{cor}

In order to prove Theorem~\ref{thm:CK} we need the following proposition:

\begin{prop}\label{prop:extensionCK}
Let $K$ and $L$ be compact spaces such that $K\sub L$. Suppose that:
\begin{enumerate}
\item[(i)] There is a bounded linear operator $E: C(K) \to C(L)$ such that $E(f)|_K=f$ for all $f\in C(K)$.
\item[(ii)] $C(K)$ has the lattice-lifting property.
\end{enumerate}
Then $K$ is a neighborhood retract of~$L$.
\end{prop}
\begin{proof}
Let $R:C(L)\to C(K)$ be the surjective lattice homomorphism given by $R(g):=g|_K$ for all $g\in C(L)$. 
Since $R\circ E =id_{C(K)}$ and $C(K)$ has the lattice-lifting property, 
there is a lattice homomorphism $S:C(K)\to C(L)$ such that $R \circ S=id_{C(K)}$ (Proposition~\ref{prop:EquivalencesLLP}). 
By the general representation of lattice homomorphisms between spaces of continuous functions on compact spaces
(see, e.g., \cite[Theorem~2.34]{AB}), there exist a continuous map $w:L\to [0,\infty)$ and a map $\pi:L \to K$ which
is continuous on~$L\setminus w^{-1}(\{0\})$ such that $S$ is given by 
$$
	S(f)(y) = w(y) f(\pi(y))
	\quad
	\text{for all $y\in L$ and all $f\in C(K)$.} 
$$
The fact that $R\circ S = id_{C(K)}$ means that $w(x) f(\pi(x)) = f(x)$ for all $x\in K$ and all $f\in C(K)$. 
Clearly, this implies that $\pi$ is the identity on~$K$ and that $w(x)= 1$ for all $x\in K$. Thus, $K$ is a retract of the open set~$L\setminus w^{-1}(\{0\})$.
\end{proof}

\begin{proof}[Proof of Theorem~\ref{thm:CK}]
Write $L:=B_{C(K)^*}$ (equipped with the $w^*$-topology) and consider the bounded linear operator $E: C(K) \to C(L)$ given by
$$
	E(f)(\mu) := \int_K f \, d\mu 
	\quad \text{for all $f\in C(K)$ and all $\mu\in L$.}
$$
Here we identify the elements of $C(K)^*$ with regular signed measures on the Borel $\sigma$-algebra of~$K$, via
Riesz's representation theorem. Since $E(f)|_K=f$ for every $f\in K$ and $C(K)$ has the lattice-lifting property, we can apply Proposition~\ref{prop:extensionCK}
to deduce that $K$ is a neighborhood retract of~$L$. 

Suppose now that $K$ is metrizable. If $K$ is finite, then it is easy to check that it is an AR. If $K$ is infinite, then
$L$ is an infinite-dimensional compact convex metrizable subset of the locally convex space $(C(K)^*,w^*)$ and so $L$ is homeomorphic to the Hilbert cube~$[0,1]^\mathbb{N}$ by Keller's theorem (see, e.g., \cite[Theorem~12.37]{fab-ultimo}). Bearing in mind Fact~\ref{fact:ARandANR}, we conclude that $K$ is an ANR. 
\end{proof}

We do not know whether the metrizability assumption can be removed from the second statement of Theorem~\ref{thm:CK}:

\begin{prob}\label{prob:metrizability}
Let $K$ be a compact space such that $C(K)$ has the lattice-lifting property. Is $K$ an ANR?
\end{prob}

\begin{rem}\label{rem:path-connected}
If $K$ is a compact space such that $C(K)$ has the lattice-lifting property, then $K$ is a finite union of pathwise connected
closed subsets.
\end{rem}
\begin{proof} As before, we consider $B_{C(K)^*}$ equipped with the $w^*$-topology. 
By Theorem~\ref{thm:CK}, there exist an open set $K\sub V \sub B_{C(K)^\ast}$ and a continuous function $r: V \to K$ such that $r|_K=id_K$. 
For each $x\in K$ we choose a convex open set $V_x \sub B_{C(K)^*}$ such that $x\in V_x \sub \overline{V_x} \sub V$. 
Since $K$ is compact, there exist finitely many $x_1,\dots,x_n\in K$ such that $K \sub \bigcup_{i=1}^n V_{x_i}$. 
Then $K=\bigcup_{i=1}^n r(\overline{V_{x_i}})$. Each $\overline{V_{x_i}}$ is compact and pathwise connected (in fact, it is convex) and so 
$r(\overline{V_{x_i}})$ is compact and pathwise connected as well.
\end{proof}

\begin{example}\label{example:c}
Let $c$ be the Banach lattice of all convergent sequences of real numbers (with the coordinatewise order).
\begin{enumerate}
\item[(i)] $c$ fails the lattice-lifting property. Indeed, $c$ is lattice isometric to~$C(\mathbb N \cup \{\infty\})$, 
where $\mathbb N \cup \{\infty\}$ denotes the one-point compactification of~$\mathbb N$ equipped with the discrete topology.
The only connected non-empty subsets of $\mathbb N \cup \{\infty\}$ are the singletons, so Remark~\ref{rem:path-connected} implies
that $c$ fails the lattice-lifting property.
\item[(ii)] For each $n\in \nat$, let $\phi_n \in Hom(c,\mathbb R)$ be the $n$th-coordinate projection. Since 
the sequence $(\phi_n)$ separates the points of~$c$, this space satisfies the $\sigma$-bounded chain condition
(Remark~\ref{rem:sigma-bcc}). Therefore, the converse of Corollary~\ref{cor:necessary} does not hold in general.
\item[(iii)] We have a short exact sequence of lattice homomorphisms
$$
	0\To c_0 \stackrel{j}{\To} c \stackrel{L}{\To} \mathbb R \To 0,
$$
where $j$ is the canonical embedding and $L$ is the map that takes each sequence to its limit. On the one hand,
$\mathbb R$ is $1$-projective and so it has the isometric lattice-lifting property. On the other hand,
$c_0$ has the isometric lattice-lifting property as a consequence of Theorem~\ref{t:unconditional} below
(cf.~\cite{AMR20}). This shows that the isometric lattice-lifting property is not a 3-space property. 
\end{enumerate}
\end{example}

The following example and corollary show, in particular, that the lattice-lifting property is not stable under taking 
sublattices.

\begin{example}\label{example:torus}
{\em Let $\mathbb T \sub \mathbb{R}^2$ be the unit circumference. Then $C(\mathbb T)$ has the lattice-lifting property
and contains a sublattice failing the lattice-lifting property.} Indeed, the first statement follows from
Corollary~\ref{cor:CK} because $\mathbb T$ is an ANR, since it is a neighborhood retract of $[-1,1]^2$ (recall Fact \ref{fact:ARandANR}).
Let $K:=\mathbb T ^\nat$ (with the product topology). 
Observe that $K$ is metrizable but it is not an ANR. Indeed, if it were, as an infinite countable product of non-empty separable metric spaces, 
all its factors but finitely many would be AR (see, e.g., \cite[proof of Theorem~1.5.8]{vanMill}). However, $\mathbb T$ is not an AR since it is not 
a retract of $[-1,1]^2$ (see, e.g., \cite[Theorem~3.5.5]{vanMill}). From Theorem~\ref{thm:CK} it follows that $C(K)$ fails the lattice-lifting property.
We will show that $C(K)$ is lattice isometric to a sublattice of~$C(\mathbb T)$. To this end, it suffices to
check that there is a continuous surjection $f: \mathbb T \to K$, because in this case the composition map $T_f: C(K) \to C(\mathbb T)$ given by
$T_f(g):=g \circ f$ for all $g\in C(K)$ would be an isometric lattice embedding.   
Now, the existence of such an~$f$ follows from the fact that $[0,1]^{\mathbb N}$ is a continuous image of $[0,1]$, by the Hahn-Mazurkiewicz theorem
(see, e.g., \cite[Theorem~31.5]{willard}), bearing in mind that $\mathbb T$ is a continuous image of $[0,1]$ and vice versa. 
\end{example}

\begin{cor}\label{cor:sublatticefails}
Let $E$ be a Banach space with $dim(E)\geq 2$. Then $FBL[E]$ contains a sublattice failing the lattice-lifting property.
\end{cor}
\begin{proof}
Let $E_0\sub E$ be a two-dimensional subspace. Since $E_0$ is complemented in~$E$, we know
that $FBL[E_0]$ is lattice isomorphic to a sublattice of $FBL[E]$ (see \cite[proof of Corollary~2.8]{ART18}).
But $FBL[E_0]$ is lattice isomorphic to $C(\mathbb T)$, see \cite[Proposition~5.3]{dPW15}. 
The conclusion follows from Example~\ref{example:torus}.
\end{proof}

For any Banach lattice~$X$, the set $Hom(X,\mathbb R)$ is $w^*$-closed in~$X^*$. Throughout the rest of this section
we discuss some topological properties of the $w^*$-compact set $Hom(X,\mathbb R) \cap B_{X^*}$ when $X$ has the isometric lifting-property.
We first introduce further terminology concerning homomorphisms on free Banach lattices.
As we already mentioned, given a Banach space~$E$, 
the set $Hom(FBL[E],\mathbb R)$ consists of all functionals $\varphi_{x^*}:FBL[E]\to \mathbb R$ defined by
$\varphi_{x^*}(f):=f(x^*)$ for every $f\in FBL[E]$, where $x^*\in E^*$. For any set $A \sub E^*$ we write 
$$
	\Phi_A:=\{\varphi_{x^*}:x^*\in A\} \sub Hom(FBL[E],\mathbb{R}).
$$
Clearly, the adjoint $\delta_E^* : FBL[E]^* \to E^*$ satisfies $\delta_E^*(\varphi_{x^*})=x^*$ for every $x^* \in E^*$ and its
restriction $\delta_E^* |_{Hom(FBL[E],\mathbb R)}: Hom(FBL[E],\mathbb R) \to E^*$ is a $w^*$-to-$w^*$ continuous bijection. In particular, we have:

\begin{rem}\label{RemNaturalHomomorphism}
Let $E$ be a Banach space and $K\sub E^*$ be a $w^*$-compact set. Then the map 
$x^* \mapsto \varphi_{x^*}$ defines a homeomorphism between $(K,w^*)$ and $(\Phi_K, w^*)$.
\end{rem}	

\begin{prop}\label{prop:Xlattice-lifting}
Let $X$ be a Banach lattice having the lattice-lifting property. Then $(\Phi_{Hom(X,\mathbb R)}, w^*)$ is a retract of $(Hom(FBL[X],\mathbb R), w^*)$. 
If in addition $X$ has the isometric lattice-lifting property, then $(Hom(X,\mathbb R) \cap B_{X^*},w^*)$ is a retract of $(B_{X^*},w^*)$.
\end{prop}
\begin{proof}
Let $T:X \to FBL[X]$ be a lattice homomorphism such that $\beta_X \circ T= id_X$. Note that for each $x^* \in Hom(X,\mathbb R)$ we have 
\begin{equation}\label{eqn:betax}
	x^*\circ \beta_X  = \varphi_{x^*} 
\end{equation}
(both are lattice homomorphisms and $x^* \circ \beta_X \circ \delta_X = x^* = \varphi_{x^*}\circ \delta_X$)
and therefore 
\begin{equation}\label{eqn:betaxx}
	(\beta_X^* \circ T^*)(\varphi_{x^*})=
	\beta_X^*(\varphi_{x^*}\circ T) = \beta_X^*(x^*\circ \beta_X \circ T)=\beta_X^*(x^*)=\varphi_{x^*}.
\end{equation}
For any $y^*\in X^*$ we have $(\beta_X^* \circ T^*)(\varphi_{y^*}) \in \Phi_{Hom(X,\mathbb R)}$, because
$T^*(\varphi_{y^*}) \in Hom(X,\mathbb R)$ and \eqref{eqn:betax} yields
\begin{equation}\label{eqn:betaxxx}
	(\beta_X^* \circ T^*)(\varphi_{y^*}) = \varphi_{T^*(\varphi_{y^*})} .
\end{equation}
So, the restriction of~$\beta_X^*\circ T^*$ is a retraction from $(Hom(FBL[X],\mathbb R),w^*)$ 
onto~$(\Phi_{Hom(X,\mathbb R)},w^*)$.

Suppose now that $X$ has the isometric lattice-lifting property, so that $T$ can be chosen in such a way that $\|T\|=1$. 
For each $y^*\in B_{X^*}$ we have $\varphi_{y^*}\in B_{FBL[X]^*}$, hence $T^*(\varphi_{y^*}) \in Hom(X,\mathbb R)\cap B_{X^*}$ 
and so $(\beta_X^* \circ T^*)(\varphi_{y^*})$ belongs to~$\Phi_{Hom(X,\mathbb R) \cap B_{X^*}}$ (by~\eqref{eqn:betaxxx}).
Thus, the restriction of~$\beta_X^*\circ T^*$ is a retraction from $(\Phi_{B_{X^*}},w^*)$ 
onto~$(\Phi_{Hom(X,\mathbb R)\cap B_{X^*}},w^*)$. 
From Remark~\ref{RemNaturalHomomorphism} it follows that $(Hom(X,\mathbb R) \cap B_{X^*},w^*)$ is a retract of~$(B_{X^*},w^*)$.
\end{proof}

A topological space $T$ is said to be {\em locally pathwise connected} if for each $t\in T$ and each open set~$U$ containing~$t$ 
there is an open set~$V$ containing~$t$ such that for every $t'\in V$ there is a continuous map $h:[0,1]\to U$ such that $h(0)=t$
and $h(1)=t'$. This property is inherited by retracts (see, e.g., \cite[Proposition~10.2]{Hu})
and holds for any convex subset of a locally convex space. So, from the previous proposition we get:

\begin{cor}\label{cor:locallypathconnected}
Let $X$ be a Banach lattice. Then $(Hom(X,\mathbb R)\cap B_{X^*},w^*)$ is locally pathwise connected
whenever $X$ has the isometric lattice-lifting property.
\end{cor}

Another application of Proposition~\ref{prop:Xlattice-lifting} is the following:

\begin{cor}\label{cor:AR}
Let $X$ be a separable Banach lattice. Then $(Hom(X,\mathbb R)\cap B_{X^*},w^*)$ is an AR
whenever $X$ has the isometric lattice-lifting property.
\end{cor}
\begin{proof} Note that $(B_{X^*},w^*)$ is homeomorphic to either $[0,1]^n$ for some $n\in \mathbb N$ (if $X$ is finite-dimensional) 
or $[0,1]^\mathbb{N}$ (if $X$ is infinite-dimensional) by Keller's theorem (see, e.g., \cite[Theorem~12.37]{fab-ultimo}). 
By Fact~\ref{fact:ARandANR} and Proposition~\ref{prop:Xlattice-lifting}, $(Hom(X,\mathbb R) \cap B_{X^*},w^*)$ is an AR.  
\end{proof}

We wonder whether the separability assumption can be removed from the previous statement:

\begin{prob}
Let $X$ be a Banach lattice. Is $(Hom(X,\mathbb R)\cap B_{X^*},w^*)$ an AR whenever $X$ has the isometric lattice-lifting property?
\end{prob}

\section{Banach spaces with a $1$-unconditional basis}\label{section:unconditional}

A Banach space with a $1$-unconditional basis becomes naturally a Banach lattice when equipped with the coordinatewise order.
The purpose of this section is to prove the following: 

\begin{thm}\label{t:unconditional}
Let $X$ be a Banach space with a $1$-unconditional basis. If we consider $X$ as a Banach lattice with the coordinatewise order, then $X$ has the isometric lattice-lifting property. 
\end{thm}

The proof of Theorem~\ref{t:unconditional} requires some auxiliary lemmata. Recall
that the free Banach lattice $FBL(A)$ over a non-empty set $A$ can be identified with $FBL[\ell_1(A)]$, see \cite[Corollary~2.9]{ART18}.
For each $a\in A$, we denote by $e_a \in \ell_1(A)$ the vector defined by $e_a(a')=0$ for all $a'\in A\setminus \{a\}$ and $e_a(a)=1$,
and we write $\delta_a:=\delta_{\ell_1(A)}(e_a) \in FBL(A)$ for simplicity.

\begin{lem}\label{lem:rojo}
Let $E$ be a Banach space. The vector lattice homomorphism
$$
	\varphi: FBL(B_E) \to \mathbb{R}^{E^*},
	\quad
	\varphi(f)(x^*):=f(x^*|_{B_E}),
$$ 
satisfies the following properties:
\begin{enumerate}
\item[(i)] $\varphi(FBL(B_E)) \sub FBL[E]$.
\item[(ii)] $\varphi$ has norm~$1$ and dense range as a bounded linear operator from $FBL(B_E)$ to $FBL[E]$.
\end{enumerate}
\end{lem}
\begin{proof}
Fix $f\in FBL(B_E)$. Clearly, $\varphi(f)$ is positively homogeneous. Given finitely many $x_1^*,\dots,x_n^*\in E^*$
with $\sup_{x\in B_E}\sum_{k=1}^n |x_k^*(x)|\leq 1$, we have
$$
	\sum_{k=1}^n |\varphi(f)(x_k^*)|=
	\sum_{k=1}^n |f(x_k^*|_{B_E})| \leq \|f\|_{FBL(B_E)},
$$ 
so $\varphi(f)\in H_0[E]$ and $\|\varphi(f)\|_{FBL[E]}\leq \|f\|_{FBL(B_E)}$. 
Therefore, $\varphi$ is a lattice homomorphism from $FBL(B_E)$ to $H_0[E]$
with $\|\varphi\|\leq 1$.

Let $S$ be the vector sublattice
of $FBL(B_E)$ generated by $\{\delta_{x}:x\in B_E\}$. Since $\varphi(\delta_{x})=\delta_E(x)$ for every $x\in B_E$
(which implies that $\|\varphi\|=1$) and $\varphi$ is a lattice homomorphism, 
$\varphi(S)$ is the vector sublattice of $H_0[E]$ generated by $\{\delta_E(x):x\in B_E\}$. Hence $\varphi(S)$ is a dense vector sublattice 
of $FBL[E]$. Since $S$ is dense in $FBL(B_E)$ and $\varphi$ is continuous, we conclude that
$\varphi(FBL(B_E))$ is a dense vector sublattice of~$FBL[E]$.
\end{proof}

\begin{defn}\label{defn:depending}
Let $E$ be a Banach space, $B \sub E$ be a non-empty set and $f: E^*\to \mathbb{R}$ be a function.  
We say that $f$ {\em depends on coordinates of~$B$} if $f(x^*)=f(y^*)$ for every $x^*,y^*\in E^*$ such that $x^*|_{B}=y^*|_{B}$.
\end{defn}

\begin{lem}\label{lem:PWgeneral}
Let $E$ be a Banach space and $f:E^*\to \mathbb{R}$ be a positively homogeneous function such that:
\begin{enumerate}
\item[(i)] $f|_{B_{E^*}}$ is norm-continuous;
\item[(ii)] there is a finite non-empty set $B \sub E$ such that $f$ depends on coordinates of~$B$. 
\end{enumerate}
Then $f\in FBL[E]$.
\end{lem}

Lemma~\ref{lem:PWgeneral} will be obtained as a consequence of a particular case
via Lemma~\ref{lem:rojo}. We isolate such particular case in Lemma~\ref{lem:PW} below. It is essentially based on the description 
of the space $FBL(A)$ for a finite set~$A$ given in~\cite[Proposition~5.3]{dPW15}, but we include a detailed proof for the reader's convenience.

\begin{lem}\label{lem:PW}
Let $A$ be a non-empty set and $f:\ell_\infty(A)\to \mathbb{R}$ be a positively homogeneous function such that:
\begin{enumerate}
\item[(i)] $f|_{B_{\ell_\infty(A)}}$ is $w^*$-continuous; 
\item[(ii)] there is a finite non-empty set $A_0 \sub A$ such that $f$ depends on coordinates of~$\{e_a: a\in A_0\}$.
\end{enumerate}
Then $f\in FBL(A)$.
\end{lem}
\begin{proof}
Let $r:\ell_\infty(A) \to \ell_\infty(A)$ be the bounded linear operator given by
$$
	r(x^*):=x^*\chi_{A_0}
	\quad 
	\text{for all $x^*\in \ell_\infty(A)$.}
$$
Note that a function $g:\ell_\infty(A)\to \mathbb{R}$ depends on coordinates of~$\{e_a: a\in A_0\}$ if and only if $g=g\circ r$.
Observe also that 
$$
	S:=\{x^*\in r(\ell_\infty(A)): \, \|x^*\|_{\ell_\infty(A)}=1\}
$$ 
is a $w^*$-compact subset of $\ell_\infty(A)$ (because $A_0$ is finite).

{\em Claim~1. If $g:\ell_\infty(A)\to \mathbb{R}$ is a positively homogeneous function which depends on coordinates of~$\{e_a: a\in A_0\}$,
then $\|g\|_{FBL[E]}\leq \sup_{x^*\in S}|g(x^*)| \cdot |A_0|$}.
In order to prove this, we suppose that $C:=\sup_{x^*\in S}|g(x^*)|<\infty$ (otherwise the inequality is trivial).
Take $x_1^*,\dots,x_k^*\in \ell_\infty(A)$ such that $\sup_{a\in A}\sum_{i=1}^k|x_i^*(a)|\leq 1$.
For each $i\in \{1,\dots,k\}$ we choose $a_i\in A_0$ such that
$|x_i^*(a_i)|=\|r(x_i^*)\|_{\ell_\infty(A)}$ and we take $y_i^*\in S$ such that $r(x_i^*)=|x_i^*(a_i)|\, y_i^*$.
Since $g=g\circ r$ and is positively homogeneous, we have
\begin{align*}
\sum_{i=1}^k|g(x_i^*)|&=
\sum_{i=1}^k|g(r(x_i^*))|=
\sum_{i=1}^k |g(y_i^*)| \cdot |x_i^*(a_i)|
\leq C \sum_{i=1}^k |x_i^*(a_i)|
\\ 
& = C \sum_{a\in A_0} \sum_{\substack{i=1\\ a_i=a}}^k |x_i^*(a)|
\leq  C \sum_{a\in A_0} \sum_{i=1}^k |x_i^*(a)| \leq C \cdot |A_0|.
\end{align*} 
This shows that $\|g\|_{FBL[E]}\leq C \cdot |A_0|$ and finishes the proof of {\em Claim~1}.

Recall that each $g\in FBL(A)$ is $w^*$-continuous on bounded subsets of~$\ell_\infty(A)$, see \cite[Lemma~4.10]{ART18}.
Let $j: FBL(A) \to C(S)$ be the restriction map, that is,
$$
	j(g):=g|_S 
	\quad
	\text{for all $g\in FBL(A)$.}
$$
Then $j$ is a lattice homomorphism with $\|j\|\leq 1$.

Let $V$ be the set of all $f\in FBL(A)$ depending on coordinates of~$\{e_a:a\in A_0\}$.
Clearly, $V$ is a sublattice of~$FBL(A)$ and $\delta_{a}\in V$ for all $a\in A_0$. Therefore, 
the sublattice $U$ generated by $\{\delta_{a}:a\in A_0\}$ is contained in~$V$.

{\em Claim~2. The equality $j(U)=C(S)$ holds.}
Indeed, since $j$ is a lattice homomorphism, $j(U)$ is a vector sublattice of~$C(S)$. By the inclusion~$U\sub V$ and {\em Claim~1}, 
$j(U)$ is norm-closed and $j|_U$ is an isomorphism between $U$ and $j(U)$
as Banach spaces. So, by the Stone-Weierstrass theorem, in order to prove that $j(U)=C(S)$
it suffices to check that $j(U)$ separates the points of~$S$ and that the constant function~$1 \in C(S)$ belongs to~$j(U)$.
On the one hand, if $x^*,y^*\in S$ are distinct, then $x^*(a)\neq y^*(a)$ for some $a\in A_0$, that is, 
$j(\delta_{a})(x^*)\neq j(\delta_{a})(y^*)$. Hence $j(U)$ separates the points of~$S$.
On the other hand, let $\varphi:=\bigvee_{a\in A_0}|\delta_{a}|\in U$. Then $j(\varphi)=\bigvee_{a\in A_0}|j(\delta_{a})|$
and so for every $x^*\in S$ we have
$$
	j(\varphi)(x^*)=\max_{a\in A_0}|j(\delta_{a})|(x^*)=\max_{a\in A_0}|x^*(a)|=1,
$$
that is, $j(\varphi)$ is the constant function~$1 \in C(S)$. The proof of {\em Claim~2} is finished.

Finally, let $f:\ell_\infty(A)\to \mathbb{R}$ be as in the statement of the lemma. By {\em Claim~2}, there is $g\in U$ such that $f|_S=g|_S$.
Since $f$ and $g$ are positively homogeneous and depend on coordinates of~$\{e_a: a\in A_0\}$, we have $f=g$, 
as it can be easily checked. Thus $f\in FBL(A)$. 
\end{proof}

\begin{proof}[Proof of Lemma~\ref{lem:PWgeneral}]
Write $B=\{x_1,\dots,x_n\}$. We can assume without loss of generality that $B \sub B_E$ and that
$B$ is linearly independent. The latter condition ensures that we can find $x_1^*,\dots,x_n^*\in E^*$
such that $x_i^*(x_j)=1$ if $i=j$ and $x_i^*(x_j)=0$ if $i\neq j$ for every $i,j \in \{1,\dots,n\}$. Define $g:\ell_\infty(B_E) \to \mathbb R$ by 
$$
	g(\xi):=f\big(\xi(x_1)x_1^*+\dots+\xi(x_n)x_n^*\big)
	\quad
	\mbox{for all }\xi\in \ell_\infty(B_E). 
$$
Clearly, g is positively homogeneous, depends on coordinates of $\{e_{x_1},\dots,e_{x_n}\}$ and 
the restriction $g|_{B_{\ell_\infty(B_E)}}$ is $w^*$-continuous (because $f$ is norm-continuous on bounded sets), so $g\in FBL(B_E)$ by Lemma~\ref{lem:PW}. 
Since $f$ depends on coordinates of $\{x_1,\dots,x_n\}$, we have 
$$
	g(x^*|_{B_E})=f\big(x^*(x_1)x_1^*+\dots+x^*(x_n)x_n^*\big)=f(x^*)
	\quad
	\mbox{for all }x^*\in E^*,
$$ 
and so we can apply Lemma~\ref{lem:rojo} to conclude that $f\in FBL[E]$. 
\end{proof}

\begin{lem}\label{lem:1unconditionalbasis}   
Let $E$ be a Banach space with a $1$-unconditional basis $(e_n)_{n\in \mathbb N}$ and let $(e_n^*)_{n\in \mathbb N}$ be the sequence in~$E^*$ of 
biorthogonal functionals associated with~$(e_n)_{n\in \mathbb N}$. Then
$$
	\left\|\sum_{i=1}^l |x_{i}^\ast(e_{m_i})| \, e_{m_i}^\ast\right\|_{E^*} \leq \sup_{x \in B_{E}}\sum_{i = 1}^l|x_i^*(x)|
$$
for all $l\in \nat$, all $x_1^*, x_2^*, \ldots, x_l^*\in B_{E^\ast}$ and all $m_1, m_2, \ldots, m_l \in \nat$.
\end{lem}
\begin{proof} Consider $m:=\max\{m_i: i=1,\dots,l \}$. Fix $z\in B_E$. We will
show that
$$
	\left|\sum_{i=1}^l |x_{i}^\ast(e_{m_i})| \, e_{m_i}^\ast(z)\right| \leq \sup_{x \in B_{E}}\sum_{i = 1}^l|x_i^*(x)|=:\alpha.
$$
For every choice of signs $\varepsilon=(\varepsilon_j)_{j=1}^m \in \{-1,1\}^m$, define $z^\varepsilon := \sum_{j=1}^m \varepsilon_j e_j^*(z) e_j$, 
so that $z^\varepsilon\in B_E$ (since the basis is $1$-unconditional) and therefore
$$
	\sum_{i=1}^l \left|\sum_{j=1}^m \varepsilon_j e_j^*(z) x_i^*(e_j) \right|
	=\sum_{i = 1}^l |x_i^\ast(z^\varepsilon)| 
	\leq \alpha.
$$
So, in order to finish the proof it suffices to check that
\begin{equation}\label{eqn:1unconditional}
	\left|\sum_{i=1}^l |x_{i}^\ast(e_{m_i})| \, e_{m_i}^\ast(z)\right| \leq 
	\max \set{\sum_{i=1}^l \left|\sum_{j=1}^m \varepsilon_j e_j^*(z) x_i^*(e_j) \right| : \, \varepsilon \in \set{-1,1}^m}.
\end{equation}
To this end, observe first that
\begin{equation}\label{eqn:triangle} 
	\max \set{\sum_{i=1}^l \left|\sum_{j=1}^m \varepsilon_j e_j^*(z) x_i^*(e_j) \right| : \, \varepsilon \in \set{-1,1}^m} 
 	\geq  
 	\frac{1}{2^m}\sum_{\varepsilon \in \set{-1,1}^m}
 	\sum_{i=1}^l \left|\sum_{j=1}^m \varepsilon_j e_j^*(z) x_i^*(e_j) \right| 
\end{equation}
\begin{eqnarray*}
 	& = &\frac{1}{2^m} \sum_{i=1}^l \sum_{\substack{\varepsilon \in \set{-1,1}^m \\ \varepsilon_{m_i}=1}} \left|\sum_{j=1}^m \varepsilon_j e_j^*(z) x_i^*(e_j) \right|
 	+\frac{1}{2^m} \sum_{i=1}^l \sum_{\substack{\varepsilon \in \set{-1,1}^m \\ \varepsilon_{m_i}=-1}} \left|\sum_{j=1}^m \varepsilon_j e_j^*(z) x_i^*(e_j) \right|
 	\\
 	&\geq &\frac{1}{2^m} \sum_{i=1}^l   
 	\left| \sum_{\substack{\varepsilon \in \set{-1,1}^m \\ \varepsilon_{m_i}=1}} \sum_{j=1}^m \varepsilon_j e_j^*(z) x_i^*(e_j) \right|
 	+ \frac{1}{2^m} \sum_{i=1}^l  
 	\left| \sum_{\substack{\varepsilon \in \set{-1,1}^m \\ \varepsilon_{m_i}=-1}} \sum_{j=1}^m \varepsilon_j e_j^*(z) x_i^*(e_j) \right|.
\end{eqnarray*}
Now, for each $i=1\dots,l$ and each $\sigma\in \{-1,1\}$ we have
\begin{multline*}
	 \left| \sum_{\substack{\varepsilon \in \set{-1,1}^m \\ \varepsilon_{m_i}=\sigma}} \sum_{j=1}^m \varepsilon_j e_j^*(z) x_i^*(e_j) \right|
	 \\=
	 \left| \sigma 2^{m-1} e_{m_i}^*(z) x_{i}^*(e_{m_i}) + \sum_{\substack{\varepsilon \in \set{-1,1}^m \\ \varepsilon_{m_i}=\sigma}} 
	 \sum_{\substack{j=1 \\ j \neq m_i}}^m \varepsilon_j e_j^*(z) x_i^*(e_j) \right|=
	 2^{m-1} | e_{m_i}^*(z) x_{i}^*(e_{m_i}) |.
\end{multline*}
Thus, from~\eqref{eqn:triangle} it follows that
$$
	\max \set{\sum_{i=1}^l \left|\sum_{j=1}^m \varepsilon_j e_j^*(z) x_i^*(e_j) \right| : \, \varepsilon \in \set{-1,1}^m}
	\geq \sum_{i=1}^l | e_{m_i}^*(z) x_{i}^*(e_{m_i}) |
	\geq \left|\sum_{i=1}^l |x_{i}^\ast(e_{m_i})| \, e_{m_i}^\ast(z)\right|,
$$
and so \eqref{eqn:1unconditional} holds. The proof is finished.
\end{proof}

We are now ready to prove Theorem~\ref{t:unconditional}.

\begin{proof}[Proof of Theorem~\ref{t:unconditional}]
We consider $X$ equipped with the coordinatewise order induced by a fixed $1$-unconditional basis $(e_n)_{n \in \mathbb{N}}$ of~$X$.
Clearly, we can assume without loss of generality that $\|e_n\|_X=1$ for all $n\in \nat$.
We will show that there is a disjoint sequence $(f_n)_{n \in \mathbb{N}} \in FBL[X]_+$ such that
$\beta_X(f_n)=e_n$ for all $n\in\mathbb N$ and
$$
	\norm{\sum_{n=1}^N a_n f_n}_{FBL[X]} \leq \norm{\sum_{n=1}^N a_n e_n}_{X}
$$ 
for every $N\in\nat$ and all $a_1,\dots,a_N\in\mathbb{R}$.
Once we have this, it is easy to check that 
$$
	T:X\rightarrow FBL[X],
	\quad T(x):=\sum_{n\in \mathbb{N}} e_n^*(x) f_n, 
$$
is a well-defined lattice homomorphism such that $\beta_X \circ T= id_X$ and $\|T\|=1$, where 
$(e_n^*)_{n\in \mathbb{N}}$ is the sequence in~$X^*$ of biorthogonal functionals associated with~$(e_n)_{n\in \mathbb{N}}$.

We divide the proof into several steps.

\smallskip
{\em Step~0. Identification of~$\beta_X$.} For each $n\in \nat$, both $\varphi_{e_n^*}$ and $e_n^*\circ \beta_X$ belong to $Hom(FBL[X],\mathbb R)$
and satisfy $\varphi_{e_n^*}\circ \delta_X = e_n^* = e_n^*\circ \beta_X \circ \delta_X$, hence $\varphi_{e_n^*}=e_n^*\circ \beta_X$.
It follows that for every $f\in FBL[X]$ we have
\begin{equation}\label{eqn:beta}
	\beta_X(f)=\sum_{n\in \nat}f(e_n^*)e_n,
\end{equation}
the series being unconditionally convergent.

\smallskip
{\em Step~1. Construction of the $f_n$'s.} 
Let $(M_n)_{n \in \mathbb{N}}$ and $(N_n)_{n \in \mathbb{N}}$ be two strictly increasing sequences 
such that $M_n < N_n$ for all $n\in \mathbb{N}$ and 
$$
	\sum_{n\in \mathbb{N}}\frac{1}{M_n}<\infty.
$$
So, each subseries of $\sum_{n\in \nat}\frac{1}{M_n}e_n$ is absolutely convergent in~$X$.  
For each $m\in \mathbb{N}$, let $g_m:[0,\infty)\to [0,1]$ be a continuous decreasing function such that $g_m(t)=0$ if $t\geq N_m$, while $g_m(t)=1$ if $t\leq M_m$. 
Given $r \in \mathbb{R}$, we write $r^+ = r\vee 0$ to denote its positive part.
For each $n \in \mathbb{N}$, we define $f_n: X^* \to [0,\infty)$ by
$$
	f_n(x^*) := 
	\begin{cases}
	\big(|x^*(e_n)|-N_n\max\set{|x^*(e_m)| : m<n}\big)^+ \cdot \prod_{m > n}g_{m}\big(|x^*(e_m)|/|x^*(e_n)|\big) & \text{if $x^*(e_n)\neq 0$,} \\
	0 & \text{if $x^*(e_n)= 0$.}
	\end{cases}
$$
Here $\prod_{m > n}g_{m}\big(|x^*(e_m)|/|x^*(e_n)|\big)$ is the limit of the sequence 
$(\prod_{m = n+1}^{n+k} g_{m}\big(|x^*(e_m)|/|x^*(e_n)|\big)_{k\in \nat}$, which is decreasing and contained in~$[0,1]$
(because each $g_m$ takes values in $[0,1]$). Clearly, $f_n$ is positively homogeneous.

\smallskip
{\em Step~2. $(f_n)_{n\in \nat}$ is disjoint.} Indeed, fix $n<l$ in~$\nat$ and pick $x^* \in X^*$ such that $f_l(x^*) \neq 0$. 
If $x^*(e_n)=0$, then $f_n(x^*)=0$ by the very definition. Suppose $x^*(e_n)\neq 0$.
Since $f_l(x^*) \neq 0$, we have $|x^*(e_l)| > N_l \max \set{|x^*(e_m)| : m<l}$ and so $|x^*(e_l)| > N_l |x^*(e_n)|$. 
Hence $g_{l}(|x^*(e_l)|/|x^*(e_n)|)=0$ and so $f_n(x^*)=0$. This shows that $f_n \wedge f_l = 0$.

\smallskip
{\em Step~3. $f_n \in FBL[X]$ for every $n \in \mathbb{N}$.} To prove this, fix $n,k\in \nat$ and define 
$h_{n,k}: X^* \to \mathbb{R}$ by 
$$
	h_{n,k}(x^*) := 
	\begin{cases}
	\big(|x^*(e_n)|-N_n\max\set{|x^*(e_m)| : m<n}\big)^+ \cdot \prod_{m=n+1}^{n+k} g_{m}\big(|x^*(e_m)|/|x^*(e_n)|\big) & \text{if $x^*(e_n)\neq 0$,} \\
	0 & \text{if $x^*(e_n)= 0$.}
	\end{cases}
$$

We claim that $h_{n,k} \in FBL[X]$. Indeed, observe that $h_{n,k}$ is positively homogeneous
and depends on coordinates of the finite set $\{e_1,\dots,e_{n+k}\}\sub X$. Further, it is easy to check that
$h_{n,k}$ is $w^*$-continuous: bear in mind that the map $x^*\mapsto |x^*(e_m)|$ is $w^*$-continuous for every $m\in \nat$, 
that all the $g_m$'s are continuous and that the inequality $|h_{n,k}(x^*)|\leq |x^*(e_n)|$ holds for every $x^*\in X^*$.
From Lemma~\ref{lem:PWgeneral} it follows that $h_{n,k}\in FBL[X]$, as claimed.

We next prove that  
\begin{equation}\label{eqn:freenorm}
	\|h_{n,k}-f_n\|_{FBL[X]} \leq \sum_{j>n+k} \frac{1}{M_{j}}.
\end{equation}
Indeed, take $x_1^*, \ldots, x_l^* \in X^\ast$ such that $\sup_{x \in B_{X}}\sum_{i = 1}^l|x_i^*(x)| \leq 1$.
We seek to estimate 
$$
	\sum_{i=1}^l |h_{n,k}(x_i^*)-f_n(x_i^*)|,
$$ 
so we can assume that $h_{n,k}(x_i^*) \neq f_n(x_i^*)$ for all $i=1,\dots,l$.
This means that for each $i = 1,\ldots, l$ we have $|x_{i}^*(e_n)| \neq 0$ and 
there is $m_i > n+k$ such that $g_{m_i}(|x_{i}^\ast(e_{m_i})|/|x_{i}^\ast(e_n)|) \neq 1$, 
that is, $|x_{i}^*(e_{m_i})|>M_{m_i} |x_{i}^*(e_n)|$. So, bearing in mind that $h_{n,k} \geq f_n$, we have
$$
	|h_{n,k}(x_i^*)-f_n(x_i^*)| = h_{n,k}(x_i^*)-f_n(x_i^*)  \leq h_{n,k}(x_i^*) \leq |x_i^*(e_n)| < \frac{1}{M_{m_i}} |x_{i}^*(e_{m_i})|
$$
for every $i = 1,\ldots, l$. It follows that
\begin{eqnarray*}
	\sum_{i=1}^l |h_{n,k}(x_i^*)-f_n(x_i^*)| & \leq & \sum_{i=1}^l \frac{1}{M_{m_i}}|x_{i}^*(e_{m_i})| 
	= \left(\sum_{i=1}^l |x_{i}^\ast(e_{m_i})| e_{m_i}^\ast\right)\left(\sum_{j>n+k}\frac{1}{M_{j}}e_{j}\right) \\
	& \leq & \left\|\sum_{i=1}^l |x_{i}^\ast(e_{m_i})| e_{m_i}^\ast\right\|_{X^*}\cdot\left \|\sum_{j>n+k}^\infty\frac{1}{M_{j}}e_j\right\|_X
	\leq \sum_{j>n+k}^\infty\frac{1}{M_{j}}, 
\end{eqnarray*}
where the last inequality is a consequence of Lemma~\ref{lem:1unconditionalbasis}.
This works for any finite collection of functionals $x_1^*, \ldots, x_l^* \in X^\ast$ 
satisfying $\sup_{x \in B_{X}}\sum_{i = 1}^l|x_i^*(x)| \leq 1$, hence inequality~\eqref{eqn:freenorm} holds.

Finally, since each $h_{n,k}$ belongs to~$FBL[X]$ and $\lim_{k\to \infty}\|h_{n,k}-f_n\|_{FBL[X]}=0$ 
(by~\eqref{eqn:freenorm} and the fact that $\sum_{j\in\nat}\frac{1}{M_j}<\infty$), we conclude
that $f_n\in FBL[X]$, which finishes the proof of {\em Step~3}.

\smallskip
{\em Step~4. $\beta_X(f_n)=e_n$ for every $n\in \nat$.} Indeed, this follows from~\eqref{eqn:beta}, bearing in mind that, for each $j\in \nat$, 
we have $f_n(e_j^*)=1$ if $j=n$ and $f_n(e_j^*)=0$ if $j\neq n$ (by the very definition of~$f_n$). 

\smallskip
{\em Step~5. The inequality}
\begin{equation}\label{eqn:normspan}
	\norm{\sum_{n=1}^N a_n f_n}_{FBL[X]} \leq \norm{\sum_{n=1}^N a_n e_n}_{X}
\end{equation}
{\em holds for every $N\in\nat$ and all $a_1,\dots,a_N\in\mathbb{R}$.}
Indeed, since the sequence $(f_n)_{n\in \nat}$ is disjoint and $(e_n)_{n\in \nat}$ is $1$-unconditional, 
the norms appearing in~\eqref{eqn:normspan} do not change by switching the signs of the coefficients, and we can 
assume that $a_n\geq 0$ for all $n=1,\dots,N$. Set $f:=\sum_{n=1}^N a_nf_n$. 
Fix $x_1^*, x_2^*, \ldots, x_k^* \in  X^\ast$ with $\sup_{x \in B_{X}} \sum_{j=1}^k |x_j^*(x)| \leq 1$, and
suppose without loss of generality that $f(x_j^*)\neq 0$ for all $j=1,\dots,k$.
Since the $f_n$'s are disjoint, for each $j=1,\ldots,k$ there is a unique $n_j \in \{1,\ldots,N\}$ such that $f_{n_j}(x_j^*) \neq 0$, hence
$$ 
	\sum_{j=1}^k |f (x_j^*)| = \sum_{j=1}^k f (x_j^*) =
	\sum_{j=1}^k a_{n_j}f_{n_j}(x_j^*).
$$
Since $f_{n_j}(x_j^*) \leq | x_{j}^*(e_{n_j})| $ for every $j=1,\ldots, k$, we get
\begin{align*}
	\sum_{j=1}^k |f (x_j^*)| & \leq  \sum_{j=1}^k  a_{n_j} |x_{j}^*(e_{n_j})| = \left(\sum_{j=1}^k |x_{j}^\ast(e_{n_j})|e_{n_j}^\ast\right)\left(\sum_{n=1}^N a_n e_n\right)\\
	& \leq \left\|\sum_{j=1}^k |x_{j}^\ast(e_{n_j})|e_{n_j}^\ast\right\|_{X^*}\cdot \left\|\sum_{n=1}^N a_n e_n\right\|_X 	\leq \left\|\sum_{n=1}^N a_n e_n\right\|_X,
\end{align*}
where the last inequality follows from Lemma~\ref{lem:1unconditionalbasis}.
This shows that~\eqref{eqn:normspan} holds.

The proof of Theorem~\ref{t:unconditional} is complete.
\end{proof}

\section*{Acknowledgements}
We wish to thank Valentin Ferenczi for suggesting this line of research.

Research partially supported by {\em Fundaci\'{o}n S\'{e}neca} [20797/PI/18] and {\em Agencia Estatal de Investigaci\'{o}n} [MTM2017-86182-P to 
A. Avil\'es, G. Mart\'inez-Cervantes, and J. Rodr\'iguez; MTM2016-76808-P, MTM2016-75196-P to P. Tradacete; all grants being cofunded by ERDF, EU].
The research of G. Mart\'inez-Cervantes was co-financed by the European Social Fund and the Youth European Initiative under {\em Fundaci\'{o}n S\'{e}neca}
[21319/PDGI/19]. P.~Tradacete gratefully acknowledges support by {\em Ministerio de Econom\'{\i}a, Industria y Competitividad} through 
{\em ``Severo Ochoa Programme for Centres of Excellence in R\&D''} [SEV-2015-0554] and {\em Consejo Superior de Investigaciones Cient\'ificas} 
through  {\em ``Ayuda extraordinaria a Centros de Excelencia Severo Ochoa''} [20205CEX001].

\end{document}